\documentclass[reqno,12pt]{amsart}%
\usepackage[body={17cm,22cm}]{geometry}

\usepackage[T1]{fontenc}
\usepackage{graphicx}
\usepackage{mathrsfs}
\usepackage[french]{babel}
\usepackage{amscd,latexsym,amsthm,amsfonts,amssymb,amsmath,amsxtra}
\usepackage[colorlinks, urlcolor=blue,  citecolor=blue]{hyperref}
\usepackage[all]{xy}%
\providecommand{\U}[1]{\protect\rule{.1in}{.1in}}
\RequirePackage{amsmath}
\RequirePackage{amssymb}

\newcommand{\BZ}{{\mathbb {Z}}}

\newcommand{\CH}{{\mathcal {H}}}

\newcommand{\CT}{{\mathcal {T}}}

\newcommand{\Hom}{{\mathrm{Hom}}}

\renewcommand{\Im}{{\mathrm{Im}}}

\newcommand{\Ker}{{\mathrm{Ker}}}

\newcommand{\Spec}{{\mathrm{Spec}}}

\newcommand{\iso}{\stackrel{\sim}{\rightarrow} }

\numberwithin{equation}{section}

\theoremstyle{remark}

\newtheorem{defi}{\rm{\textbf{D\'efinition}}}[section]

\newtheorem{rem}[defi]{\rm{\textbf{Remarque}}}

\theoremstyle{plain}

\newtheorem{thm}[defi]{\rm{\textbf{Th\'eor\`eme}}}

\newtheorem{cor}[defi]{\rm{\textbf{\textbf{Corollaire}}}}

\newtheorem{prop}[defi]{\rm{\textbf{\textbf{Proposition}}}}

\begin{document}


\title[Corrigendum]{Corrigendum \`a ``Sous-groupe de Brauer invariant et obstruction de descente it\'er\'ee'' Algebra \& Number Theory 14 (2020), no. 8, 2151--2183.}

\author{Yang CAO}

\address{Yang CAO \newline 	University of Science and Technology of China,
\newline 96 Jinzhai Road,
 \newline	230026 Hefei, Chine}

 \email{yangcao1988@ustc.edu.cn; yangcao1988@gmail.com}

\date{\today.}


\maketitle

\begin{abstract} Le \S 2 de notre travail \cite{C4} utilise  un \'enonc\'e sur la formule
de  K\"unneth de degr\'e 2 provenant de   l'article  \cite{SZ}. Il \'et\'e remarqu\'e (\cite{SZ1})
que cet \'enonc\'e de \cite{SZ} n'est pas correct.
Nous corrigeons \cite[\S 2]{C4}.  Le reste de \cite{C4} reste correct.
\end{abstract}

\tableofcontents

Tous les r\'esultats de \cite{C4} sont corrects sauf certains  r\'esultats dans \cite[\S 2]{C4}.
Dans \cite[\S 2]{C4}, les \'enonc\'es \cite[Prop. 2.2, Cor. 2.4 et Cor. 2.7]{C4} sont corrects, et ce  sont   les seuls r\'esultats que nous utilisons dans le reste de \cite{C4}.
Mais la formule de K\"unneth de degr\'e 2 (\cite[Prop. 2.6]{C4}) n'est pas correcte.   
Sa d\'emonstration repose sur \cite[Prop. 2.2]{SZ},  qui n'est pas correcte : voir   \cite[Rem. 1.2]{SZ1} pour un contre-exemple. 
Plus pr\'ecis\'ement, le morphisme not\'e par $-\circ -$ dans le premier diagramme commutatif du preuve de \cite[Cor. 2.3]{C4} n'est pas un isomorphisme. 
De ce fait,
 il y a des erreurs dans \cite[Cor. 2.3]{C4} et donc dans \cite[Lem. 2.5 et Prop. 2.6]{C4}.
Dans ce corrigendum, nous corrigeons les \'enonc\'es \cite[Prop. 2.6]{C4} (remplac\'e par le th\'eor\`eme \ref{BiDKthm} ci-dessous) et
  aussi la d\'emonstration de \cite[Cor. 2.7]{C4} (remplac\'e par le corollaire \ref{BiDKcor2.3.2} ci-dessous).

\bigskip

Dans toute cette note,  $k$ est un corps quelconque, $k_s$ une cl\^oture s\'eparable de $k$ et $\Gamma_k:=\mathrm{Gal}(k_s/k)$. Sauf  mention explicite du contraire, une vari\'et\'e est une $k$-vari\'et\'e.
Fixons un entier $n\geq 2$ avec $\mathrm{char}(k)\nmid n$.
Soient $U$, $V$ deux $k$-vari\'et\'es g\'eom\'e\-triquement int\`egres.
On consid\`ere le diagramme commutatif
\begin{equation}\label{BiDdiagprodu}
\xymatrix{U\times_kV \ar[r]^-{p_2}\ar[d]^{p_1} & V\ar[d]^{q_2}\\
U\ar[r]^-{q_1}& \Spec\ k.
}
\end{equation}

\section{Rappels}

Nous rappelons la notion de torseur universel de $n$-torsion, la formule de K\"unneth g\'en\'erale et l'homomorphisme $\varepsilon$ dans \cite[\S 5]{SZ} et dans \cite[(2.9)]{C4}.

\medskip

Soit $X$ une $k$-vari\'et\'e g\'eom\'etriquement int\`egre.
Soit $S_X$  un $k$-groupe fini commutatif tel que $n\cdot S_X=0$ et que $S_X^*:=\Hom_{k_s}(S_X,\mu_n)\cong H^1(X_{k_s},\mu_n)$ comme $\Gamma_k$-modules.
Pour  la notion de \emph{torseur universel de $n$-torsion pour $X$}, nous renvoyons \`a    \cite[D\'ef. 2.1]{C4}. 
La propri\'et\'e fondamentale d'un torseur universel de $n$-torsion est la proposition \ref{BiDprop2.1} ci-dessous.

\begin{prop}(\cite[Prop. 2.2]{C4})\label{BiDprop2.1}
Soit $\CT_X$ un torseur universel de $n$-torsion pour $X$.
Soit $S$ un $k$-groupe fini commutatif tel que $n\cdot S=0$.
Alors, pour tout $S$-torseur $Y$ sur $X$, il existe un unique homomorphisme $\phi: S_X\to S $ tel que
$$\phi_*([\CT_X])-[Y]\in \Im (H^1(k,S)\to H^1(X,S)). $$
\end{prop}

L'homomorphisme $\phi$ dans la proposition \ref{BiDprop2.1} est appel\'e \emph{le n-type de $[Y]$}.
Il induit un isomorphisme de $\Gamma_k$-modules:
\begin{equation}\label{BiDprop2.1e3}
\tau_{X,S}: \Hom_{k_s}(S_X,S) \to H^1(X_{k_s},S): \phi\mapsto \phi_*([\CT_X]).
\end{equation}
Par exemple, $\tau_{X,\mu_n}: \Hom_{k_s}(S_X,\mu_n) \iso H^1(X_{k_s},\mu_n)$ est un isomorphisme.

Rappelons que $H^1(X_{k_s},S)$ est fini.

De plus, si $X(k)\neq \emptyset$, alors pour chaque $x\in X(k)$
il existe alors un unique torseur universel de $n$-torsion $\CT_X$ pour $X$ tel que $x^*[\CT_X]=0\in H^1(k,S_X)$ (cf. \cite[p. 2157]{C4}).

\bigskip

Consid\'erons le diagramme commutatif (\ref{BiDdiagprodu}).

Rappelons la formule de K\"unneth dans \cite{SGA4.5}.

Soit $m$ un entier avec $m|n$. Soit $D^-(k)$ (resp. $D^+(k)$) la cat\'egorie d\'eriv\'ee  born\'ee \`a droite (resp. \`a gauche) de la cat\'egorie des $\BZ/m$-faisceaux \'etales sur le petit site de $\Spec\ k$.

Soient $M$, $N$ deux $k$-groupes finis commutatifs tels que $M(k_s),N(k_s)$ soient des $\BZ/m$-modules plats.
 Alors $M\otimes^L_{\BZ/m} N=M\otimes_{\BZ/m} N$ et le cup-produit donne un quasi-isomorphisme (\cite[Th. finitude Cor. 1.11]{SGA4.5} ou \cite[Cor. 9.3.5]{Fu}):
 $$\cup: \ Rq_{1,*}M\otimes^L_{\BZ/m} Rq_{2,*}N \cong  R(q_1\circ p_1)_* (M\otimes_{\BZ/m} N)$$
dans $D^-(k)$.
Ceci induit le cup-produit (\cite[Prop. 6.4.12]{Fu}):
$$\cup_j: \oplus_{r+s=j} R^rq_{1,*}M\otimes_{\BZ/m} R^sq_{2,*}N \to \CH^j(Rq_{1,*}M\otimes^L_{\BZ/m} Rq_{2,*}N) \iso R^j(q_1\circ p_1)_* (M\otimes_{\BZ/m} N).$$

Dans le cas o\`u $m=p$ est un nombre premier, puisque $\BZ/p$ est un corps, tout $\BZ/p$-module est plat et $-\otimes_{\BZ/p}^L-=-\otimes_{\BZ/p}-$. Le cup-produit ci-dessus induit un isomorphisme pour tout $j$:
 \begin{equation}\label{BiDKcupprod}
 \cup_j: \oplus_{r+s=j} R^rq_{1,*}M\otimes_{\BZ/p} R^sq_{2,*}N \iso  R^j(q_1\circ p_1)_* (M\otimes_{\BZ/p} N). 
 \end{equation}

\bigskip

Soit $S$ un $k$-groupe fini commutatif avec $n\cdot S=0$.

En appliquant \cite[Th. finitude Thm. 1.9 (ii)]{SGA4.5} au diagramme (\ref{BiDdiagprodu}) on obtient :
\begin{equation}\label{BiDKbasechange}
Rp_{1,*}p_2^*(S)\cong q_1^*Rq_{2,*}(S).
\end{equation}
En fait, pour appliquer ce th\'eor\`eme, on choisit $X=V$, $Y=S=Spec\ k$, $S'=U$ pour les $X,Y,S,S'$ dans \cite[Th. finitude Thm. 1.9]{SGA4.5}.

\bigskip

Rappelons maintenant l'homomorphisme $\varepsilon$ de Skorobogatov et Zarhin.

Supposons qu'il existe des torseurs universels de $n$-torsion $\CT_U$ pour $U$ (sous le groupe $S_U$) et $\CT_V$ pour $V$ (sous le groupe $S_V$).
Skorobogatov et Zarhin introduisent un homomorphisme (\cite[\S 5]{SZ}):
\begin{equation}\label{BiDvarepsilion}
\varepsilon': \Hom_k(S_U\otimes_{\BZ/n} S_V,S)\to H^2(U\times V,S): \phi\mapsto \phi_*([\CT_U]\cup [\CT_V]),
\end{equation}
o\`u $\cup$ est le cup-produit $H^1(U,S_U)\times H^1(V, S_V)\to H^2(U\times V,S_U\otimes_{\BZ/n} S_V)$.

Dans le cas o\`u $S:=M\otimes_{\BZ/n}N\cong M\otimes_{\BZ/m}N$ (puisque $m|n$), on a le diagramme commutatif:
\begin{equation}\label{BiDKcommudiag}
\xymatrix{\Hom(S_U,M)\otimes_{\BZ/n}\Hom(S_V,N)\ar[r]^-{\xi}\ar[d]^{\tau_{U,M}\otimes \tau_{V,N}}_{\cong}&\Hom(S_U\otimes_{\BZ/n}S_V, M\otimes_{\BZ/n}N)\ar[r]^-= &\Hom(S_U\otimes_{\BZ/n}S_V, S)\ar[d]^{\varepsilon'}\\
  H^1(U,M)\otimes_{\BZ/n}H^1(V,N) \ar[r]^-{\cong}  & H^1(U,M)\otimes_{\BZ/m}H^1(V,N)\ar[r]^-{\cup_2}& H^2(U\times V,S)
}
\end{equation}
o\`u $\tau_{U,M}\otimes \tau_{V,N}$ est induit par (\ref{BiDprop2.1e3}) qui est donc un isomorphisme, $\cup_2$ est induit par (\ref{BiDKcupprod}) 
et $\xi(\phi_1,\phi_2)=\phi_1\otimes_{\BZ/n}\phi_2$.

Ce diagramme est commutatif car, pour tout $\phi_1\in \Hom(S_U,M),\phi_2\in \Hom(S_V,N)$, on a
$$\cup_2\circ (\tau_{U,M}\otimes \tau_{V,N}) (\phi_1,\phi_2  )=\phi_{1,*}[\CT_U]\cup \phi_{2,*}[\CT_V]\stackrel{(1)}{=}(\phi_1\otimes\phi_2)_*([\CT_U]\cup [\CT_V])
=\varepsilon'(\phi_1\otimes\phi_2)=(\varepsilon'\circ \xi)(\phi_1,\phi_2),  $$
o\`u (1) d\'ecoule du diagramme commutatif
$$ \xymatrix{H^1(U,S_U)\ar[d]^{\phi_{1,*}}\ar@{}[r]|-{\times}& H^1(V,S_V)\ar[d]^{\phi_{2,*}} \ar[r]^-{\cup}& H^2(U\times V,S_U\otimes_{\BZ/n}S_V)\ar[d]^{(\phi_1\otimes\phi_2)_*}\\
H^1(U,M)\ar@{}[r]|-{\times}& H^1(V,N) \ar[r]^-{\cup}& H^2(U\times V,S).
} $$

\begin{rem}\label{BiDKrem}
Dans le cas o\`u $S=\mu_n$, le $\varepsilon'$ dans (\ref{BiDvarepsilion}) est essentialement le $\varepsilon$ dans \cite[\S 5]{SZ} et dans \cite[(2.9)]{C4}.
Plus pr\'ecis\'ement, $S^*_V:=\Hom_{k_s}(S_V,\mu_n)$ et
$$\varepsilon: \Hom_k(S_U,S_V^*)\to H^2(U\times V,\mu_n): \phi\mapsto \eta_*(\phi_*[\CT_U]\cup [\CT_V]),$$
o\`u $\eta: S_V^*\otimes_{\BZ/n} S_V\to \mu_n$ est le morphisme d'\'evaluation.
Dans ce cas, on a 
$$ \varepsilon \circ \varphi=\varepsilon',$$ 
o\`u $\varphi: \Hom_k(S_U\otimes_{\BZ/n} S_V,\mu_n)\iso \Hom_k(S_U,\Hom_{k_s}(S_V,\mu_n))=\Hom_k(S_U,S_V^*)$ 
est l'isomorphisme canonique induit par les foncteurs adjoints $-\otimes_{\BZ/p}S_V$ et $\Hom_{k_s}(S_V,-)$.

Ceci vaut car, pour tout $\phi\in \Hom_k(S_U,S_V^*)$, on a $\varphi^{-1}(\phi)=\eta\circ (\phi\otimes id_{S_V})$ et donc 
$$\varepsilon (\phi)=\eta_*(\phi_*[\CT_U]\cup [\CT_V])=\eta_*((\phi\otimes id_V)_*([\CT_U]\cup [\CT_V]))=\varphi^{-1}(\phi)_*([\CT_U]\cup [\CT_V])=(\varepsilon'\circ \varphi^{-1})(\phi).$$
\end{rem}

\section{Th\'eor\`eme principal}

Voici la version correcte de la formule de K\"unneth de degr\'e 2, 
qui remplace \cite[Prop. 2.6]{C4}.

\begin{thm}\label{BiDKthm}
Supposons que $k$ est s\'eparablement clos.
Soient $U$, $V$ deux $k$-vari\'et\'es g\'eom\'e\-triquement int\`egres et $S$ un $\BZ/n$-module fini (vu comme un $k$-groupe commutatif).
Consid\'erons le diagramme (\ref{BiDdiagprodu}) et l'homomorphisme $\varepsilon'$ dans (\ref{BiDvarepsilion}) ci-dessus.
 Alors on a des isomorphismes:
$$\psi^1(S):\ \ H^1(U,S)\oplus H^1(V,S)\xrightarrow{(p_1^*,p_2^*)} H^1(U\times V,S), $$
$$\psi^2(S):\ \ H^2(U,S)\oplus H^2(V,S)\oplus \Hom(S_U\otimes_{\BZ/n} S_V,S)\xrightarrow{(p_1^*,p_2^*,\varepsilon')}H^2(U\times V,S),$$
et un homomorphisme injectif $\psi^3(S):\ \ H^3(U,S)\oplus H^3(V,S)\xrightarrow{(p_1^*,p_2^*)} H^3(U\times V,S).$
\end{thm}

\begin{proof}
Puisque $k$ est s\'eparablement clos, $U(k)\neq\emptyset$ et $V(k)\neq\emptyset$. 
Fixons deux points $u\in U(k),v\in V(k)$.
 L'existence des sections 
 $$l_v: U\times v\to U\times V,\ \ \ \ l_u: u\times V\to U\times V$$
  de $p_1,p_2$ impliquent que
 l'homomorphisme $(p_1^*,p_2^*):\ H^i(U,S)\oplus H^i(V,S)\to H^i(U\times V,S)$ est injectif et scind\'e pour tout $i\geq 1$.
 Soit $$H^2_{prim}(S):=\Ker(H^2(U\times V,S)\xrightarrow{(l_v^*,l_u^*)}H^2(U,S)\oplus H^2(V,S)).$$
 Puisque $[\CT_U]|_u=0, [\CT_V]|_v=0$, on a $l_u^*\circ \varepsilon'=0$, $l_v^*\circ \varepsilon'=0$ et donc $\Im(\varepsilon')\subset H^2_{prim}(S).$
 
Consid\'erons la suite spectrale $E_2^{i,j}=R^iq_{1,*}R^jp_{1,*}(S)\Rightarrow H^{i+j}(U\times V,S)$. On a:

(i) Par d\'efinition, l'homomorphisme $$H^i(U,S)=R^iq_{1,*}(S)=E_2^{i,0}\to H^{i}(U\times V,S)$$ est exactement $p_1^*:H^i(U,S)\to H^{i}(U\times V,S)$, qui est injectif et scind\'e.

(ii) Puisque 
\begin{equation}\label{BiDKeq2}
Rp_{1,*}(S)=Rp_{1,*}p_2^*(S)\stackrel{(\ref{BiDKbasechange})}{\cong} q_1^*Rq_{2,*}(S)=q_1^*R\Gamma(V,S),
\end{equation}
tout $R^jp_{1,*}(S)$ est le faisceau constant $H^j(V,S)$.
La composition 
$$H^j(V,S)\to H^j(U\times V,S)\to E_2^{0,j}=q_{1,*} (H^j(V,S))=H^j(V,S)$$
est l'identit\'e, puisqu'elle  est l'identit\'e sur $u\times V$.
Donc $H^j(U\times V,S)\to E_2^{0,j}$ est surjectif et scind\'e.

(iii) D'apr\`es (\ref{BiDprop2.1e3}), $E_2^{1,1}=H^1(U, H^1(V,S))\cong \Hom(S_U,\Hom(S_V,S))\cong \Hom(S_U\otimes_{\BZ/n}S_V,S)$,
qui est un groupe fini. Ainsi $|E_2^{1,1}|=|\Hom(S_U\otimes_{\BZ/n}S_V,S)|.$

D'apr\`es (i) et (ii),
$$H^1(U\times V,S)\cong E_2^{1,0}\oplus E_2^{0,1},\ \ \ H^2(U\times V,S)\cong E_2^{2,0}\oplus E_2^{0,2}\oplus E_2^{1,1},$$
 et, dans ces sommes directes, $(p_1^*,p_2^*)$ induit les isomorphismes $H^j(U,S)\oplus H^j(V,S)\cong E_2^{j,0}\oplus E_2^{0,j}$ pour $j=1,2$.
Alors $H^2_{prim}(S)\cong E_2^{1,1}$ et, d'apr\`es (iii), $|\Hom(S_U\otimes_{\BZ/n}S_V,S)|=|H^2_{prim}(S)|$.
 Ainsi $\varepsilon': \Hom(S_U\otimes_{\BZ/n}S_V,S)\to H^2_{prim}(S)$ est un isomorphisme si et seulement si $\varepsilon'$ est injectif.
 
On a donc montr\'e :

 (iv) $\psi^1(S)$ est un isomorphisme et $\psi^3(S)$ est injectif.
 
 (v) $\psi^2(S)$ est un isomorphisme si et seulement si $\psi^2(S)$ est injectif, si et  seulement si $\varepsilon'$ est injectif.  
 
 \medskip
 
 Dans le cas o\`u  $|S|$ est un nombre premier $p$, on a $S\cong \BZ/p$ et
on consid\`ere le diagramme (\ref{BiDKcommudiag}) avec $M\cong N\cong \BZ/p$.
Dans ce diagramme, $\xi$ est un isomorphisme et, d'apr\`es (\ref{BiDKcupprod}), $\cup_2$ est injectif. 
 Alors $\varepsilon'$ est injectif et, d'apr\`es (v), $\psi^2(S)$ est un isomorphisme.

 Dans le cas g\'en\'eral,   par r\'ecurrence sur  $|S|$, on peut supposer que $\psi^2(S')$ est un isomorphisme pour tout $\BZ/n$-module fini $S'$ avec $|S'|<|S|$.
Alors si $|S|$ n'est pas un nombre premier, il existe une suite exacte $0\to S_1\to S\to S_2\to 0$ de $\BZ/n$-modules finis avec $S_1,S_2\neq 0$.
Ceci induit une suite exacte 
$$0\to \Hom(S_U\otimes_{\BZ/n} S_V,S_1)\to \Hom(S_U\otimes_{\BZ/n} S_V,S)\to \Hom(S_U\otimes_{\BZ/n} S_V,S_2)$$
 et donc un diagramme commutatif de suites exactes:
$$\xymatrix{H^1(U,S_2)\oplus H^1(V,S_2)\ar[r]^-{f_1}\ar[d]^{\psi^1(S_2)}&F(S_1)\ar[r]\ar[d]^{\psi^2(S_1)}&F(S)\ar[r]\ar[d]^{\psi^2(S)}&
F(S_2)\ar[d]^{\psi^2(S_2)}\\
H^1(U\times V,S_2)\ar[r]&H^2(U\times V,S_1)\ar[r]&H^2(U\times V,S)\ar[r]&H^2(U\times V,S_2).
}$$
o\`u $F(S):=H^2(U,S)\oplus H^2(V,S)\oplus \Hom(S_U\otimes_{\BZ/n} S_V,S)$ et $f_1$ est la composition
$$H^1(U,S_2)\oplus H^1(V,S_2)\xrightarrow{(\partial_U,\partial_V,0)} H^2(U,S_1)\oplus H^2(V,S_1)\oplus \Hom(S_U\otimes_{\BZ/n} S_V,S_1) = F(S_1) .$$
 Puisque $|S_1|, |S_2|<|S|$, par r\'ecurrence, $\psi^2(S_1),\psi^2(S_2)$ sont des isomorphismes.
D'apr\`es (iv), $\psi^1(S_2)$ est un isomorphisme. Par (v) et le lemme des cinq, $\psi^2(S)$ est injectif et donc un isomorphisme, ce qui \'etablit le th\'eor\`eme.
\end{proof}

 Soient $U$, $V$ deux vari\'et\'es g\'eom\'e\-triquement int\`egres sur $k$.
On consid\`ere le diagramme commutatif (\ref{BiDdiagprodu}) ci-dessus.
 Soit $S$ un $k$-groupe fini commutatif avec $n\cdot S=0$.

Si $U(k)\neq\emptyset$, alors il existe un torseur universel de $n$-torsion pour $U$.
Un point $u\in U(k)$ induit $u^*: Rq_{1,*}S\to S$ et un morphisme surjectif $u^*: H^i(U,S)\to H^i(k,S)$ pour tout $i$.
Notons
$$H^i_u(U,S):=\Ker (H^i(U,S)\xrightarrow{u^*}H^i(k,S)).$$

\begin{cor}\label{BiDKcor2.3.2} 
Sous les notations et hypoth\`eses ci-dessus,
supposons que $U(k)\neq\emptyset$, et qu'il existe des torseurs universels de $n$-torsion $\CT_U, \CT_V$ comme ci-dessus.
Soit    $u\in U(k)$.
Alors, pour tout $k$-groupe fini commutatif $S$ avec $n\cdot S=0$, on a un isomorphisme:
$$H^2_u(U,S)\oplus H^2(V,S)\oplus \Hom_k(S_U\otimes_{\BZ/n}S_V,S)\xrightarrow{(p_1^*,p_2^*,\varepsilon') }H^2(U\times V,S).$$
\end{cor}

\begin{proof}
Consid\'erons le morphisme dans $D^+(k)$:
$$\psi:= ( p_1^*+p_2^*, u^*+0) :Rq_{1,*}S\oplus Rq_{2,*}S \to (R(q_1\circ p_1)_*S)\oplus S$$
D'apr\`es le th\'eor\`eme \ref{BiDKthm}, le c\^one $C_{\psi}$ de $\psi$ est dans $D^{\geq 2}(k)$ et la composition
\begin{equation}\label{BiDKeq2}
\Hom_{k_s}(S_U\otimes_{\BZ/n}S_V,S)\xrightarrow{\varepsilon'|_{k_s}} H^2(U_{k_s}\times V_{k_s},S )=\CH^2 (R(q_1\circ p_1)_*S)\to \CH^2(C_{\psi})
\end{equation}
est un isomorphisme. Ceci induit une suite exacte
$$0\to  H^2_u(U,S)\oplus H^2(V,S)\xrightarrow{(p_1^*,p_2^*)} H^2(U\times V,S) \to  H^0(k,\CH^2(C_{\psi}) ).$$
D'apr\`es (\ref{BiDKeq2}), la composition  
 $$\Hom_k(S_U\otimes_{\BZ/n} S_V,S)\xrightarrow{\varepsilon'} H^2(U\times V,\mu_n)\to H^0(k,\CH^2(C_{\psi}) )$$
  est un isomorphisme, d'o\`u le r\'esultat.
\end{proof}

Le corollaire \ref{BiDKcor2.3.2} et la remarque \ref{BiDKrem} donnent directement \cite[Cor. 2.7]{C4}.

\begin{rem}
Dans \cite{Lv}, Lv utilise \cite[Prop. 2.6]{C4} pour \'etablir  \cite[Lem. 3.3]{Lv}. Dans sa d\'emonstration, \cite[Prop. 2.6]{C4} peut \^etre remplac\'e par le th\'eor\`eme \ref{BiDKthm} ci-dessus.
Donc tous les r\'esultats de \cite{Lv} restent corrects (sauf \cite[Lem. 3.2]{Lv}).
\end{rem}

\bibliographystyle{alpha}
\end{document}